\documentclass{article}
\usepackage[T1]{fontenc}
\usepackage[utf8]{inputenc}
\usepackage[british]{babel}
\usepackage{amsmath,amssymb,amsthm, xcolor}
\usepackage[all]{xy}
\usepackage{comment}

\newtheorem{theo}{Theorem}[section]
\newtheorem{prop}[theo]{Proposition}
\newtheorem{lem}[theo]{Lemma}
\newtheorem{cor}[theo]{Corollary}
\newtheorem{quest}[theo]{Question}

\theoremstyle{definition}
\newtheorem{defi}[theo]{Definition}
\newtheorem{ex}[theo]{Example}

\theoremstyle{remark}
\newtheorem{rem}{Remark}

\DeclareMathOperator{\Aut}{Aut}
\DeclareMathOperator{\Soc}{Soc}
\DeclareMathOperator{\Ret}{Ret}

\title{Chain conditions on skew braces and solutions of the Yang-Baxter Equation}
\author{M. Di Matteo \and R. Esteban-Romero\and M. Ferrara \and V. P\'erez-Calabuig}
\date{}

\begin{document}

\maketitle

\begin{abstract}
Classical works of Hall and McLain show that solubility and local nilpotency play a key role in deriving finite generation in groups from maximal or minimal conditions on normal subgroups. In this work, brace-theoretical analogues of Hall's and McLain's results are analysed for skew braces satisfying the maximal or minimal condition on ideals. We also introduce finiteness and chain conditions on non-degenerate set-theoretic solutions of the Yang-Baxter equation, and their impact on associated structure and permutation skew braces of solutions is also described.
\end{abstract}



\section{Introduction}

Skew braces are a novel algebraic structure which has its origins in the classification of non-degenerate set-theoretic solutions of the Yang-Baxter equation (solutions of the YBE, for short). The foundations of such a rich interconnection lie on the well-known fact that every skew brace provides a solution, and every solution is controlled by the solution provided by two key skew braces associated with every solution: the so-called structure skew brace and permutation skew brace of a solution. Consequently, algebraic properties of skew braces can be translated to properties of solutions, and viceversa.

The motivation of this article comes from an important fact: the structure skew brace of a solution is always infinite 
(see Proposition~\ref{prop:G(X,r)infinite}). Thus, as every algebraic theory of an infinite structure, the algebraic theory of infinite skew braces calls for the analysis of finiteness conditions in order to get a tighter grip on their internal structural nature. One of the most fruitful approaches in this sense is to consider maximal and minimal conditions on chains of substructures of a skew brace: skew braces satisfying the minimal condition on ideals (resp. subbraces), or \emph{i-artinian} (resp. \emph{s-artinian}) skew braces; and skew braces satisfying the maximal condition on ideals (resp. subbraces), or \emph{i-noetherian} (resp. \emph{s-noetherian}) skew braces. Due to the dual nature of skew braces, oscillating between groups and rings, some brace-theoretical analogues of outstanding results in ring theory and group theory have been obtained for i-artinian skew braces (see~\cite{JespersKubatVanAntwerpenVendramin21} and~\cite{i-artin}), and i-noetherian skew braces (see~\cite{i-noeth}), as well as structural results are obtained for s-artinian skew braces (see~\cite{s-artin}). Despite the progress achieved, there is still considerable work ahead, as these generalisations are by no means mere formal extensions.

In this light, Hall and McLain's seminal articles on finiteness conditions (see~\cite{Hall54} and~\cite{McLain56}) show that soluble groups and locally nilpotent groups constitute a suitable setting for the study of finiteness conditions in groups: every soluble group satisfying the maximal condition on normal subgroups is finitely generated; and for locally nilpotent groups, the maximal and the minimal condition on normal subgroups are equivalent to the maximal and the minimal condition on subgroups, respectively. 

Our first main aim in this work is to analyse brace-theoretical analogues of Hall and McLain's results within the context of soluble skew braces and nilpotency notions of skew braces, as these algebraic properties constitute an algebraic framework particularly well suited to the analysis of two of the most extensively studied properties of solutions: decomposability and multipermutability.
We prove that soluble i-noetherian skew braces are finitely generated (Theorems~\ref{teo:Hall-soluble-BRACE} and~\ref{teo:Hall-soluble-TWO-SIDEDskewbrace}), respectively for the outstanding classes of braces ---i.e. skew braces with abelian additive group--- and two-sided skew braces, which are closed to be Jacobson radical rings (see~\cite{Rump07}). As for notions of nilpotency, we show that every locally multipermutational i-noetherian skew brace is also finitely generated (Theorem~\ref{teo:McLain-local-multiperm}). For central nilpotency, which is a stronger property than multipermutability, we can say more: finitely generated centrally nilpotent skew braces are s-noetherian, so that for locally centrally nilpotent skew braces, i-noetherian, s-noetherian and being finitely generated turn out to be equivalent (Corollary~\ref{cor:McLain-local-central}). Moreover, an example shows that the converse does not hold for multipermutational skew braces (Example~\ref{ex:skbr}). 

In \cite{McLain56}, McLain has also investigated the connection between the minimal conditions on subgroups and on normal subgroups in the setting of locally nilpotent groups. The main obstacle to a complete extension of his result is that the skew brace analogue of a general property concerning finite indices in groups remains one of the most difficult open problems: given a finite index subbrace $S$ of a skew brace $B$, does its \textit{core} (the maximal ideal of $B$ contained in $S$) have finite index too? This question was also addressed in~\cite{VanAntwerpen2026} for two-sided skew braces. An alternative proof is presented here (see Lemma~\ref{lem:core_two-sided}). Thus, we prove that for locally multipermutational two-sided skew braces, i-artinian and s-artinian are equivalent properties (Theorem~\ref{theo:minid}).

The second main aim of this work is to initiate the study of finiteness conditions for solutions of the YBE, as they can be useful for the classification of infinite-dimensional general solutions of the YBE. We introduce finite generation and locality for solutions, as well as chain conditions on subsolutions and congruences, and we delve into the study of the interplaying between finiteness conditions on solutions and those on the associated structure and permutation skew braces (see Section~\ref{sec:finite-cond-sol} for further details). In this framework, special attention should be paid to the analysis of finiteness conditions for an outstanding class of solutions: the so-called Lyubashenko solutions. Our results show that, with regard to finiteness conditions, such solutions behave in a way that is remarkably close to the class of abelian groups: we prove that for a general Lyubashenko solution the maximal conditions on subsolutions and on congruences, and the minimal condition on subsolutions are all equivalent to being the solution finitely generated (see Corollary~\ref{cor:lyusol} and Theorem~\ref{teo:max-con-equiv-lyusol}). Nevertheless, the minimal condition on congruences exhibits a different behavior: every Lyubashenko solution satisfying the minimal condition on congruences must be finite (Theorem~\ref{teo:min-con-lyusol}), and there exists a one-generated Lyubashenko solution which does not satisfy the minimal condition on congruences (see Example~\ref{ex:lyusol-notmin}).

\section{Preliminaries}
This section is divided into two parts, one concerning well-known results of skew braces, and the other one concerning solutions, which contains new results.

\subsection{Skew braces}

A \emph{skew brace} is defined as a set with two group operations $(B,+,\cdot)$ such that it satisfies a (left) distributivity property: $a\cdot (b+c) = a\cdot b - a + a\cdot c$, for every $a,b,c\in B$. We write products with juxtaposition and from the distributivity property it holds that both identity elements in $(B,+)$ and $(B,\cdot)$ coincide. We denote it by~$0$. Moreover, if $\mathfrak{X}$ is a class of groups, $B$ is said to be of $\mathfrak{X}$-type if $(B,+)\in \mathfrak{X}$ --- skew braces of abelian type are also known simply as \emph{braces}. A \emph{two-sided} skew brace satisfies both previous left distributivity property and also right distributivity property: $(b+c)a = ba - a + ca$ for every $a,b,c\in B$.

Both group operations in a skew brace are related by the so-called $\lambda$-action provided by a homomorphism $\lambda\colon a \in (B,\cdot)\mapsto \lambda_a\in \Aut(B,+)$, with $\lambda_a(b) = -a+ab$ for every $b\in B$, and the so-called star operation $a \ast b= -a + ab - b = \lambda_a(b) - b$ for every $a,b\in B$. If $X,Y$ are subsets of $B$, $X\ast Y$ is the additively generated subgroup $\langle x\ast y \mid x\in X, y \in Y\rangle_+$. The importance of two-sided skew braces is that they are close to be Jacobson radical rings. Indeed, for braces $B$, $(B,+,\cdot)$ is two-sided if, and only if, $(B,+,\ast)$ is a Jacobson radical ring (see~\cite{Rump07}).

A subset $S\subseteq B$ is a \emph{skew subbrace}, or simply a \emph{subbrace}, if each $(S,+)$ and $(S,\cdot)$ are subgroups of $(B,+)$ and $(B,\cdot)$, respectively. We write $S\leq B$. A subbrace $I$ is an \emph{ideal} if  $(I,+) \unlhd (B,+)$ and $B\ast I, I\ast B \subseteq I$. It is also equivalent to say that $(I,+)\unlhd (B,+)$, $(I,\cdot)\unlhd (B,\cdot)$ and $I$ is $\lambda$-invariant. In that case, we write $I \unlhd B$, and it turns out that $bI = b+I$ for every $b\in B$, so that $(B/I,+,\cdot)$ is a skew brace. Clearly, it also holds that $\lambda_{aI}(bI) = \lambda_a(b)I$ for every $a,b\in B$. If $I$ and $J$ are ideals of $B$, such that $J\leq I$, we say that $I/J$ form a \emph{chief factor} of $B$ if $I/J$ is minimal in $B/J$. In two-sided skew braces, multiplicative conjugations are also automorphisms of the additive group (\cite[Lemma~4.1]{Nasybullov19}). Thus, characteristic subgroups of the additive group of an ideal are also ideals.

A map $f\colon B_1 \rightarrow B_2$ is a \emph{homomorphism} of skew braces if $f$ preserves sums and products. We say that $f$ is an epimorhism (resp. isomorphism) if $f$ is surjective (resp. bijective). It turns out that $\ker f = \{a\in B_1\mid f(a) = 0_{B_2}\}$ is an ideal of~$f$ and $B_1/\ker f$ is isomorphic to~$f(B_2)$.

Given $X$ a subset of $B$, we write 
\begin{itemize}
\item $\langle X\rangle = \bigcap \{S\leq B\mid X \subseteq S\}$ the subbrace generated by~$X$;
\item $X^B = \bigcap \{S \unlhd B \mid X \subseteq S\}$ the ideal generated by~$X$;
\item if $S \leq B$, $S_B$ the maximal ideal of $B$ contained in~$S$.
\end{itemize}
In all cases, if $X$ is a singleton set $\{a\}$, we write just $a$ instead of $\{a\}$. Moreover, if $X$ is finite, we say that $\langle X \rangle$ is finitely generated. If we need to emphasise generation with respect to some group operation $+$ or $\cdot$ we put the corresponding symbol next to it.

Nilpotency notions (see~\cite{CedoSmoktunowiczVendramin19} and~\cite{BallesterEstebanFerraraPerezCTrombetti25-pacificjm-cent-nilp}) are associated with two fundamental ideals of a skew brace $B$: 
\begin{align*}
\Soc(B) & = \{a \in B \mid ab = a+b = b+a, \ \forall b \in B\} = \\
& = \{a \in B \mid a \ast b = [a,b]_+ = 0, \ \forall b\in B\} = \ker \lambda \cap \mathrm{Z}(B,+)\\
\zeta(B) & = \{a \in B \mid ab = ba = a+b = b+a,\ \forall b\in B\} = \\
& = \{a\in B \mid a \ast b = b \ast a = [a,b]_+ = 0,\ \forall b\in B\} = \Soc(B) \cap \mathrm{Z}(B,\cdot)
\end{align*}
With a general construction, it is possible to define two ascending ordinal series $\{\Soc_{\alpha}(B)\}_{\alpha\in\mathcal{O}}$ and $\{\zeta_{\alpha}(B)\}_{\alpha\in\mathcal{O}}$ starting from $\Soc_0(B) = \zeta_0(B) = 0$. They are called, respectively, the \textit{socle series} and the \textit{central series}, and the last terms are called  \textit{hypersocle} and  \textit{hypercenter}, respectively. We say that 
\begin{itemize}
    \item $B$ is \emph{multipermutational} of level $n$ if $n$ is the first natural number such that $\operatorname{Soc}_n(B)=B$;
    \item $B$ is \emph{hypermultipermutational} if the hypersocle of $B$ is equal to~$B$;
    \item $B$ is \emph{locally multipermutational} if any finitely generated subbrace has finite multipermutational level.
    \item $B$ is \emph{locally hypermultipermutational} if every finitely generated subbrace is hypermultipermutational;
    \item $B$ is \emph{centrally nilpotent} if there is a finite natural number $n$ such that $\zeta_n(B)=B$;
    \item $B$ is \emph{hypercentral} if the hypercenter of $B$ is equal to~$B$;
    \item $B$ is \emph{locally centrally nilpotent} if every finitely generated subbrace is centrally nilpotent.
\end{itemize}
If $B$ is multipermutational of level $n$, then $B$ is of nilpotent type and \emph{right nilpotent}, that is, the so-called \emph{right series} of~$B$ given by
\[ B^{(1)} = B \geq B^{(2)} = B\ast B \geq \cdots \geq B^{(m)} = B^{(m-1)}\ast B \geq \cdots \]
satisfies that $B^{(n+1)} = 0$. Moreover, it holds that $B^{(n+1-k)} \leq \Soc_k(B)$ for every $0 \leq k \leq n$. It turns out that every term of the right series of $B$ is an ideal.

A skew brace is said to be \emph{abelian} if $\zeta(B) = \Soc(B) = B$. Then, \emph{soluble} skew braces are introduced in~\cite{BallesterEstebanJimenezPerezC24-solubleskewbraces}, as those ones admitting an abelian series, i.e. a series of ideals of $B$:
\[0 = I_0 \leq I_1 \leq \cdots \leq I_n = B\]
such that $I_j/I_{j-1}$ is abelian for every $1\leq j \leq n$. The \emph{derived length} of a soluble skew braces is the smallest length $n$ of such an abelian series.


\subsection{Solutions of the YBE}
A set-theoretic solution of the YBE is a pair $(X,r)$, where $X$ is a non-empty set and $r\colon X \times X \rightarrow X \times X$ is a map satisfying that 
\begin{equation}
\label{eq:braided_eq}
r_{12}r_{23}r_{12}=r_{23}r_{12}r_{23},
\end{equation}
where $r_{12}=r\times \operatorname{id}_X$ and $r_{23}=\operatorname{id}_X\times r$. We write $r(x,y) = (\lambda_x(y),\rho_y(x))$ for every $x,y\in X$, the components of~$r$. Then, $(X,r)$ is said to be non-degenerate if $\lambda_x$ and $\rho_x$ are bijective maps for any $x\in X$. According to~\cite{JedlickaPilitowska26}, every non-degenerate set-theoretic solution is bijective, and $(X,r)$ is said to be \emph{involutive} if $r^{-1} = r$. We use the word \emph{solution} for referring to non-degenerate solution.

\begin{ex}[Lyubashenko's solution]
Let $X$ be a set and let $\lambda$, $\rho$ be permutations of $X$ such that $\lambda\rho = \rho\lambda$. Then, $(X,r)$ is a solution with $r(x,y) = (\lambda(y), \rho(x))$, i.e. $\lambda_x = \lambda$ and $\rho_x = \rho$ for every $x\in X$. We say that $(X,r)$ is a \emph{Lyubashenko solution}. In particular, if $\lambda_x = \rho_x = \mathrm{id}_X$ for every $x\in X$, then $r(x,y) = (y,x)$ for every $x,y\in X$, and $(X,r)$ is also called a \emph{twist solution}.
\end{ex}

A \emph{subsolution} of a solution $(X,r)$ is given by a subset $Y\subseteq X$ such that $(Y,r_{|_{Y}})$ is still a solution, or equivalently, $r(Y\times Y)= Y\times Y$. We write $(Y,r)\leq (X,r)$. A \emph{congruence} of a solution $(X,r)$ is an equivalence relation $\sim$ on $X$, such that $(\overline{X},\overline{r})$ is still a solution of the YBE, where $\overline{X}=X/{\sim}$ and for any $x,y\in X$ 
$$\overline{r}([x]_{\sim},[y]_{\sim})=([\lambda_x(y)]_{\sim},[\rho_{y}(x)]_{\sim}).$$ 
This is equivalent to requiring that for any $x_1,x_2,y_1,y_2\in X$
\[
\begin{array}{c}
x_1\sim x_2\\
y_1\sim y_2 
\end{array} \quad \Longleftrightarrow \quad 
\begin{array}{c}
\lambda_{x_1}(y_1)\sim \lambda_{x_2}(y_2)\\
\rho_{y_1}(x_1)\sim\rho_{y_2}(x_2)
\end{array}
\]
A \emph{homomorphism} (resp. epimorphism, isomorphism) of solutions $f \colon (X,r) \rightarrow (Y,s)$ is a map (resp. surjective, bijective map) for which the equivalence relation $x\ker f y$ if, and only if, $f(x) = f(y)$ is a congruence. It is equivalent to say that the following diagram commutes:
\begin{equation}
\label{eq:diagram_hom_solutions}
\xymatrix{X \times X \ar[d]_{f\times f} \ar[r]^{r}  & X \times X  \ar[d]^{f \times f} \\ Y \times Y \ar[r]_{s} & Y \times Y}
\end{equation}
We say that $(X/\ker f, \bar{r})$ is isomorphic to the subsolution $(f(X), \left.s\right|_{f(X)})$ of $(Y,s)$. 

Every skew brace $B$ provides a solution, the so-called solution associated with~$B$: $(B, r_B)$, defined  by $r_B(a,b) = (\lambda_a(b), \rho_b(a))$ for every $a,b\in B$, where $\rho\colon b \in (B,\cdot) \mapsto \rho_b\in \mathrm{Sym}_B$ is an antihomomorphism given by 
\[ \rho_b(a) = (a^{-1}+b)^{-1}b \quad \text{for every $a,b\in B$.}\]
For instance, if $B$ is an abelian skew brace, $\lambda_a = \rho_a = \mathrm{id}_B$ for every $a\in B$, and therefore, $(B,r_B)$ is a twist solution.

On the other hand, every solution $(X,r)$ has associated the group 
\begin{equation}
\label{eq:structure_group}
G(X,r) = \langle X \mid xy = \lambda_x(y)\rho_y(x), \ \forall x,y\in X \rangle,
\end{equation}
called the structure group of $(X,r)$. We may write $G_X:= G(X,r)$ for the sake of simplicity.

\begin{prop}
\label{prop:G(X,r)infinite}
Every generator element in $G(X,r)$ is torsion free. In particular, $G(X,r)$ is an infinite group.
\end{prop}

\begin{proof}
The set of relators in $G(X,r)$ is given by $R = \{xy\rho_y(x)^{-1}\lambda_x(y)^{-1}\mid x,y\in X\}$. Then, $G(X,r) \cong F_X/N$, where $N$ is the normal subgroup generated by $R$ in the free group over~$X$: $F_X$. Observe that the exponent sum of every word in $R$ is zero, and therefore, $x^n \notin N$ for every $x\in X$ and every $n\in \mathbb{N}$. Hence, every generator element is torsion free.
\end{proof} 

Following~\cite{LuYanZhu00}, it turns out that $G(X,r)$ admits a skew brace structure $(G(X,r),+,\cdot)$. We call $G(X,r)$ the \emph{structure skew brace} of $(X,r)$, and the associated solution satisfies the following commuting diagram
\begin{equation}
\label{eq:commuting_diagram_X-G(X,r)}
\xymatrix{X \times X \ar[d]_{\iota\times \iota} \ar[r]^{r}  & X \times X  \ar[d]^{\iota \times \iota} \\ G_X \times G_X  \ar[r]_{r_{G_X}} & G_X \times G_X}
\end{equation}
where $\iota\colon x\in X \mapsto \iota(x) = x\in G(X,r)$ is a homomorphism of solutions not necessarily injective. If $\iota$ is injective, $(X,r)$ is said to be \emph{injective}. We call $\mathrm{Inj}(X,r)$ the quotient of $(X,r)$ by the kernel congruence: $x\ker\iota\, y$ if, and only if, $\iota(x) = \iota(y)$. Clearly, $G(X,r) \cong G(\mathrm{Inj}(X,r))$.

\begin{ex}
\label{ex:involutive-freeabelian}
If $(X,r)$ is involutive, $(X,r)$ is injective and it turns out that $(G(X,r),+)$ is isomorphic to the free abelian group over~$X$ (see~\cite{EtingofSchedlerSoloviev99}).
\end{ex}

Furthermore, following Theorem~4 of~\cite{LuYanZhu00}, the pair $(G(X,r),\iota)$ satisfies the following universal property: if $B$ is a skew brace and $f\colon (X,r) \rightarrow (B,r_B)$ is a homomorphism of solutions, then there exists a unique skew brace homomorphism $g\colon G_X \rightarrow B$ such that the following diagrams commute
\begin{equation}
\label{eq:universal_prop_brace}
\xymatrix{X \ar[r]^{\iota} \ar[d]_{f} & G_X \ar[dl]^{g} \\B & } \qquad \xymatrix{G_X \times G_X \ar[d]_{g\times g} \ar[r]^{r_{G_X}}  & G_X \times G_X  \ar[d]^{g\times g} \\ B \times B  \ar[r]_{r_B} & B \times B}
\end{equation}
It turns out that if $f\colon (X,r) \rightarrow (Y,s)$ is a homomorphism of solutions, then from the previous universal property applied to $\iota_Y \circ f \colon X \rightarrow G(Y,s)$, there exists a unique skew brace homomorphism $\bar f \colon G(X,r) \rightarrow G(Y,s)$ such that the following diagram commutes:
\[
\xymatrix{X \ar[d]^{\iota_X} \ar[r]^{f}  & Y \ar[d]^{\iota_Y} \\ G(X,r) \ar[r]_{\bar{f}}  & G(Y,s) }
\]

If $G(X,r)$ is the structure skew brace of a solution $(X,r)$, we can define a homomorphism of multiplicative groups
\[ \phi_{G_X} \colon G(X,r) \rightarrow \langle (\lambda_x, \rho_x^{-1}) \mid x \in X \rangle \leq \mathrm{Sym}_X\times \mathrm{Sym}_X\]
given by 
\[ w = x_1^{\varepsilon_1}\cdots x_n^{\varepsilon_n} \mapsto (\lambda_{x_1}^{\varepsilon_1}\cdots \lambda_{x_n}^{\varepsilon_n}, \rho_{x_n}^{-\varepsilon_n} \cdots \rho_{x_1}^{-\varepsilon_1}),\]
which is well-defined as equation~\eqref{eq:braided_eq} provides that $\lambda_x\lambda_y = \lambda_{\lambda_x(y)}\lambda_{\rho_y(x)}$ and $\rho_y\rho_x = \rho_{\rho_y(x)}\rho_{\lambda_x(y)}$. Following~\cite{CedoJespersKubatVanAntwerpenVerwimp23}, it turns out that $\ker \phi_{G_X}$ is an ideal of $G(X,r)$ contained in $\Soc(G(X,r))$ such that $G(X,r)/\ker\phi_{G_X}$ is a skew brace with multiplicative group isomorphic to
\[ \mathcal{G}(X,r)  = \langle (\lambda_x, \rho_x^{-1}) \mid x \in X \rangle \leq \mathrm{Sym}_X\times \mathrm{Sym}_X.\]
Thus, $\mathcal{G}(X,r) \cong G(X,r)/\ker \phi_{G_X}$ is called the \emph{permutation skew brace} of~$(X,r)$. 

\begin{rem}
\label{nota:lambda_actions}
Write $\bar G:= G(X,r)/\ker \phi_{G_X}$, and $\lambda^G$, $\lambda^{\bar{G}}$, and $\lambda^{\mathcal{G}}$ the $\lambda$-action in $G(X,r)$, $\bar G$, and $\mathcal{G}(X,r)$, respectively. Let $x,y\in X$. According to~\eqref{eq:commuting_diagram_X-G(X,r)}, it holds that $\lambda^G_x(y) = \iota(\lambda_x(y))= \lambda_x(y) \in G(X,r)$, so that
\[ \lambda^{\bar{G}}_{x\ker\phi_{G_X}}(y\ker\phi_{G_X}) = \lambda^G_x(y)\ker\phi_{G_X} = \lambda_x(y)\ker\phi_{G_X}, \quad \text{for every $x,y\in X$}.\]
Thus, $\lambda^{\mathcal{G}}_{(\lambda_x, \rho_x^{-1})}((\lambda_y, \rho_y^{-1})) = (\lambda_{\lambda_x(y)}, \rho_{\lambda_x(y)}^{-1}) \in \mathcal{G}(X,r)$. Similarly, the same holds for the $\rho$-action.
\end{rem}

If $f\colon (X,r) \rightarrow (Y,s)$ is a homomorphism of solutions, then by~\eqref{eq:diagram_hom_solutions}, it follows that $\bar{f}\colon G(X,r)\rightarrow G(Y,s)$ satisfies that $\bar{f}(\ker \phi_{G_X}) \subseteq \ker\phi_{G_Y}$. Thus, the homomorphism 
\[ \phi_{G_Y}\circ \bar{f}\colon G(X,r) \rightarrow \mathcal{G}(Y,s)\]
induces a homomorphism $\bar{\bar{f}}\colon \mathcal{G}(X,r) \rightarrow \mathcal{G}(Y,s)$.

The following easy-to-check proposition collects information about the interrelation of substructures and quotients between solutions and associated skew braces. 

\begin{prop}
\label{prop:subsol-subbrace}
Let $(X,r)$ be a solution of the YBE, $(Y,r)$ be a subsolution, and $\sim$ a congruence in~$X$. Then:
\begin{enumerate}
\item $G(Y,r_{Y})$ maps onto $\langle Y\rangle\subseteq G(X,r)$;
\item $\mathcal{G}(Y,r_Y)$ is isomorphic to a section of $\mathcal{G}(X,r)$;
\item $G(X/{\sim},r_{\sim})$ is isomorphic to a quotient of $G(X,r)$;
\item $\mathcal{G}(X/{\sim},r_{\sim})$ is isomorphic to a quotient of $\mathcal{G}(X,r)$.
\end{enumerate}
\end{prop}


The following property of solutions is naturally associated with nilpotency notions of skew braces. Let $(X,r)$ be a solution and for any $x,y\in X$, consider
\[ x \operatorname{Ret} y \quad \Longleftrightarrow \quad 
\begin{array}{c}
\lambda_{x}= \lambda_y;\\
\rho_x=\rho_y.
\end{array}\]
It follows that $\operatorname{Ret}$ is a congruence and $\operatorname{Ret}(X,r)=(X/\operatorname{Ret},r_{\operatorname{Ret}})$ is called the \emph{retraction} of $(X,r)$. We write $[x]_{\Ret}$, or simply $[x]$ if the relation $\Ret$ is understood, the equivalence class of $x$ under $\Ret$.

For every successive ordinal $m+1$,  we can recursively define $\operatorname{Ret}^{m+1}(X,r)=\operatorname{Ret}(\operatorname{Ret}^{m}(X,r))$. Instead, if $\alpha $ is a limit ordinal, take the equivalence relation $\operatorname{Ret}^{\alpha}$ defined as: $\forall x,y\in X$, $x\operatorname{Ret}^{\alpha} y$ if, and only if, there is a $\beta < \alpha$ such that $x\operatorname{Ret}^\beta y$. Then, $\operatorname{Ret}^\alpha$ is a congruence, and it defines a solution $\operatorname{Ret}^{\alpha}(X,r)=(X/\operatorname{Ret^{\alpha}},r_{\operatorname{Ret}^{\alpha}})$. For every ordinal $\alpha$, we write $[x]_{\alpha}$ the equivalence class under~$\Ret^\alpha$.

A solution is said \emph{hypermultipermutational} if there is an ordinal $\alpha$ such that $\operatorname{Ret}^\alpha(X,r)$ is the solution over a singleton set. In particular, it is said \emph{multipermutational} of level $n$, if $n$ is the smallest natural number for which $\Ret^n(X,r)$ is a solution over a singleton. For instance, every Lyubashenko solution is multipermutational of level~$1$.

Let $(X,r)$ and $(Y,s)$ be two solutions of the YBE, and let $f\colon X\rightarrow Y$ a homomorphism of solutions. Take $\pi_X \colon (X,r) \rightarrow \Ret(X,r)$ and  $\pi_Y\colon (Y,s) \rightarrow \Ret(Y,s)$ canonical epimorphisms of each solution to its correspondent quotient. Following~\cite{CedoJespersKubatVanAntwerpenVerwimp23}, $f$ induces a homomorphism of solutions $\Ret(f)\colon \Ret(X,r) \rightarrow \Ret(Y,s)$, $\Ret(f)([x]) = [f(x)]$ for every $x\in X$, such that the following diagram commutes:
\begin{equation}
\label{eq:homorphism_Ret}
\xymatrix{(X,r) \ar[d]^{\pi_X} \ar[r]^{f}  & (Y,s) \ar[d]^{\pi_Y} \\ \Ret(X,r) \ar[r]_{\Ret(f)}  & \Ret(Y,s) }
\end{equation}
Using transfinite induction, this also holds for every ordinal~$\alpha$.

If $B$ is a skew brace, then $a \in \Soc(B)$ if, and only if, $\lambda_a = \rho_a = \mathrm{id}_B$. Then, for every $a,b\in B$, $a\Soc(B) = b\Soc(B)$ if, and only if, $\lambda_a = \lambda_b$ and $\rho_a = \rho_b$. In terms of the associated solution $(B,r_B)$, for every $a,b\in B$, it holds that $a\Soc(B) = b\Soc(B)$ if, and only if, $a \Ret b$. Thus, for every ordinal $\alpha$, 
\[ (B/\Soc_\alpha(B), r_{B/\Soc_\alpha(B)}) = \Ret^\alpha(B,r_B).\]
Therefore, $B$ is hypermultipermutational (resp. multipermutational of level~$n$) if, and only if, $(B, r_B)$ is hypermultipermutational (resp. hypermultipermutational of level~$n$). This equivalence can be also translated to the structure skew brace of a solution by the following known result (see~\cite{CedoJespersKubatVanAntwerpenVerwimp23}).

\begin{theo}
\label{theo:multeq}
    Let $(X,r)$ be a solution of the Yang-Baxter Equation. Then the following statements are equivalent:
    \begin{enumerate}
        \item $(X,r)$ is a multipermutational solution;
        \item $G(X,r)$ has finite multipermutational level;
        \item $\mathcal{G}(X,r)$ has finite multipermutational level.
    \end{enumerate}
\end{theo}

We bring this section to a close by extending the previous theorem to the hypermultipermutational case.

\begin{theo}
Let $(X,r)$ be a solution of the YBE. Then, the following are equivalent:
\begin{enumerate}
\item $(X,r)$ is a hypermultipermutational solution;
\item $G(X,r)$ is a hypermultipermutational skew brace;
\item $\mathcal{G}(X,r)$ is a hypermultipermutational skew brace.
\end{enumerate}
\end{theo}

\begin{proof}
Assume that $(X,r)$ is a hypermultipermutational solution, and suppose that $\operatorname{Soc}(G(X,r))=0$. Let $\alpha$ be an ordinal. Since $\iota \colon X\rightarrow G(X,r)$ is a homomorphism of solutions, by~\eqref{eq:homorphism_Ret}, it induces a homomorphism of solutions
$$\Ret^\alpha(\iota)\colon \operatorname{Ret}^\alpha(X,r)\longrightarrow \operatorname{Ret}^\alpha(G(X,r))=\frac{G(X,r)}{\operatorname{Soc}_\alpha(G(X,r))}=G(X,r),$$
for every ordinal $\alpha$. Observe that we can write 
\[ \Ret^\alpha(\iota)(x) = [x]_{\alpha} = x \in \Ret^{\alpha}(G(X,r)), \quad \text{for every $x\in X$,}\]
as $\Ret^\alpha(G(X,r)) = G(X,r)$. 

Applying the universal property~\eqref{eq:universal_prop_brace} to $\Ret^\alpha(\iota)$, we conclude that there exists a skew brace homomorphism 
$$\overline{\Ret^{\alpha}(\iota)} \colon G(\operatorname{Ret}^\alpha(X,r))\longrightarrow G(X,r),$$ 
which is clearly surjective, as  
\[ \overline{\Ret^\alpha(\iota)}([x_1]_\alpha^{\varepsilon_1}\dots[x_n]_\alpha^{\varepsilon_n})= \Ret^\alpha(\iota)(x_1)^{\varepsilon_1}\cdots \Ret^{\alpha}(\iota)(x_n)^{\varepsilon_n} = x_1^{\varepsilon_1}\cdots x_n^{\varepsilon_n}\]
for every $x_1^{\varepsilon_1}\dots x_n^{\varepsilon_n} \in G(X,r)$

Similarly, we extend the canonical epimorphism  $\pi_{X,\alpha}\colon  (X,r)\rightarrow \operatorname{Ret}^{\alpha}(X,r)$ to a skew brace epimorphism 
$$\overline{\pi}_{X,\alpha}\colon G(X,r)\longrightarrow G(\operatorname{Ret}^\alpha(X,r)).$$ 
Clearly, it follows that $\overline{\pi}_{X,\alpha}$ and $\overline{\Ret^{\alpha}(\iota)}$ are inverse to each other. Therefore, $G(X,r)\simeq G(\operatorname{Ret}^{\alpha}(X,r))$.
 
Since $(X,r)$ is hypermultipermutational, there is an ordinal $\alpha$ such that $\Ret^\alpha(X,r) = (X/\Ret^{\alpha}, r_{\Ret^{\alpha}})$ is a singleton solution, which is clearly involutive as $r_{\Ret^{\alpha}}$ is the identity on $X/\!\Ret^\alpha \times X/\!\Ret^{\alpha}$. Thus, according to~\eqref{eq:structure_group} and Example~\ref{ex:involutive-freeabelian}, $(G(X,r),\cdot)$ and $(G(X,r),+)$ are isomorphic to an infinite cyclic group. Since $r_{\Ret^{\alpha}}$ is an identity map, according to the commuting diagram~\eqref{eq:commuting_diagram_X-G(X,r)}, it follows that the $\lambda$-action $\lambda\colon (G(X,r),\cdot)\rightarrow \Aut((G(X,r),+))$ must be trivial. Hence, $\Soc(G(X,r)) = G(X,r)$ and we arrive to a contradiction.

2 implies 3, clearly follows as $\mathcal{G}(X,r)$ is isomorphic to a quotient of $G(X,r)$ by an ideal contained in $\operatorname{Soc}(G(X,r))$.

Assume that $\mathcal{G}(X,r)$ is hypermultipermutational. According to Remark~\ref{nota:lambda_actions}, the map $f\colon X \rightarrow \mathcal{G}(X,r)$, $x \mapsto (\lambda_x, \rho_x^{-1})$ is a homomorphism of solutions, and therefore, it induces an injective homomorphism of solutions $\bar f \colon \Ret(X,r) \rightarrow \mathcal{G}(X,r)$. Applying~\eqref{eq:homorphism_Ret} recursively, it follows that $\Ret(X,r)$ is hypermultipermutational, and so is~$(X,r)$.
\end{proof}

\section{Chain conditions on ideals of skew braces}

In this section, we prove brace-theoretical analogues of Hall and McLain results about maximal and minimal conditions on normal subgroups. Recall that we call a skew brace $B$ i-noetherian (resp. s-noetherian) if it satisfies the maximal condition on ideals (resp. subbraces), and we call $B$ i-artinian (resp. s-artinian) if it satisfies the minimal condition on ideals (resp. subbraces). Throughout this section, we use the following well-known properties of the star product of elements:
\begin{align}
a \ast (b+c) & = a \ast b + b + a\ast c - b \label{eq:ast_dist_sum}\\
(ab) \ast c & = a \ast (b \ast c) + b \ast c + a \ast c \label{eq:ast_dist_prod}
\end{align}

Firstly, we focus on Hall's result: every soluble group satisfying the maximal condition on normal subgroups is finitely generated (see~\cite[page~420]{Hall54}). This result is an easy consequence of the description of the normal closure of a cyclic subgroup of an abelian normal subgroup.

The description of a one-generated subbrace is one of the most difficult open problems of the skew brace theory. However, in the class of braces some partial results are obtained in terms of nilpotency notions (see~\cite{finitely-leftnil,KurdachenkoSubbotin24}). 

\begin{rem}
Let $B$ be a skew brace and let $a\in B$. If $n\in \mathbb{N}$, $\ast_n(b_1, \ldots, b_n)(a)$ denotes the set of all possible $\ast$-products of $n+1$ elements, starting at $a$, and multiplying by each $b_i$ (either on the left or on the right) from $i=1$ to $i = n$. For instance, given $a,b_1,b_2\in B$:
\[ \ast_2(b_1,b_2)(a) = \{ (a \ast b_1) \ast b_2, b_2 \ast (a \ast b_1), (b_1\ast a)\ast b_2, b_2 \ast (b_1 \ast a)\} \]
\end{rem}

We prove here the following lemma.

\begin{lem}
\label{lem:id-generated}
Let $B$ be a  brace and let $I$ be an abelian ideal of~$B$. For every $a \in I$, we have that 
\[ a^B = \langle a, \ast_n(b_1,\ldots, b_n)(a) \mid n\in \mathbb{N}\text{ and } b_1,\ldots, b_n \in B\rangle_+. \]
\end{lem}

\begin{proof}
Let $n\in \mathbb{N}$. Given $b_1, \ldots, b_n, b\in B$, if $x\in \ast_n(b_1,\ldots, b_n)(a)$, then
\[ b \ast x, x \ast b \in \ast_{n+1}(b_1,\ldots, b_n,b)(a).\]
Call $J=\langle a, \ast_n(b_1,\ldots, b_n)(a) \mid n\in \mathbb{N}\text{ and } b_1,\ldots, b_n \in B\rangle_+$. It holds that $J \subseteq \langle a \rangle^B$ by the definition of an ideal. Thus, it suffices to see that $J$ is an ideal. 

Each element $x\in J$ can be written as $x=\sum_{t=1}^r\varepsilon_tx_t$, with $\varepsilon_t\in\{-1, 1\}$, and either $x_t = a$ or $x_t\in \ast_n(b_1,\dots, b_n)(a)$ for $1\le t\le r$. Since $\langle a \rangle^B \subseteq I$ and $I$ is abelian, by~\eqref{eq:ast_dist_sum}, it follows that $c*\sum_{t=1}^r\varepsilon_tx_t=\sum_{t=1}^r\varepsilon_t(c*x_t)\in J$ for every $c\in B$. Hence, $B\ast J \subseteq J$.

Consider now $x, y\in J \subseteq I$ and $z\in B$. Since $I$ is an abelian ideal,  $y*z\in I$ and $x*(y*z)=0$. Moreover, $x+y = xy$ and~\eqref{eq:ast_dist_prod} yields
\[(x+y)*z=(xy)*z=x*(y*z)+y*z+x*z=x*z+y*z.\]
We conclude that $\left(\sum_{t=1}^r\varepsilon_tx_t\right)*z=\sum_{t=1}^r\varepsilon_t(x_t*z)\in J$ for every $c\in B$. Hence, $J\ast B \subseteq J$, and the lemma follows.
\end{proof}

This lemma provides a generalisation of Hall's theorem for braces. 

\begin{theo}
\label{teo:Hall-soluble-BRACE}
Every i-noetherian soluble brace is finitely generated.
\end{theo}

\begin{proof}
Let $B$ be an i-noetherian soluble brace. If $B$ is abelian, the theorem trivially holds. Suppose that the derived length of $B$ is $n > 1$, and let 
\[0 = I_0 \leq I_1 \leq \cdots \leq I_n = B\]
be an abelian series of $B$. By induction, we can assume that $B/I$ is finitely generated, with $I:= I_1$. Thus, $B = SI = S+I$, where $S = \langle s_1, \ldots, s_m\rangle$ is a finitely generated subbrace.

Since $B$ is i-noetherian, we can find $a_1, \dots, a_k\in I$ such that 
\[ I = a_1^B + \ldots +  a_k^B.\]
Let $a \in I$. We claim that for every $n\in \mathbb{N}$, and each $b_1, \ldots, b_n\in B$, it holds
\[\ast_n(b_1, \ldots, b_n)(a) = \ast_n(s_1, \ldots, s_n)(a),\   \text{for certain $s_1,\ldots, s_n\in S$.}\]
Let $b\in B$, and take $b = s+x=s'x'$ with $s,s'\in S$ and $x, x'\in I$. Thus, by~\eqref{eq:ast_dist_prod}, it follows that
\[ b \ast a = (s'x') \ast a = s' \ast (x' \ast a) + x'\ast a + s' \ast a = s' \ast a,\]
as $x'\ast a = 0$ because $I$ is abelian. Analogously, $a \ast b = a \ast (s+x) = a \ast s$.

Assume the claim holds for some $n \geq 1$. Let $b_1, \ldots, b_n, b_{n+1} \in B$ and let $y \in \ast_{n+1}(b_1, \ldots, b_{n+1})(a)$. Thus, either $y = b_{n+1} \ast y'$ or $y = y' \ast b_{n+1}$ for some $y' \in \ast_{n}(b_1, \ldots, b_n)(a) \subseteq I$. By induction, $y' \in \ast_n(x_1, \ldots, x_n)$ for certain $x_1, \ldots, x_n$. Take $b_{n+1} = s+x=s'x'$ for certain $s,s' \in S$ and $x,x'\in I$. The previous paragraph yields either $y = b_{n+1} \ast y' \in \ast_{n+1}(x_1, \ldots, x_n,x)(a)$ or $y = y' \ast x \in \ast_{n+1}(x_1, \ldots, x_n,x)(a)$, respectively. Thus, the claim holds.

Finally, applying Lemma~\ref{lem:id-generated}, we have that 
\[  a_i^B = \langle a_i, \ast_{n}(x_1, \ldots, x_n)(a_i) \mid n \in \mathbb{N},\ x_1, \ldots, x_n \in S\rangle_+ \leq \langle a_i, s_1, \ldots, s_m\rangle\]
for each $1 \leq i \leq k$. Hence, $I \leq \langle s_1, \ldots, s_m, a_1, \ldots, a_k\rangle$, and therefore, $B$ is finitely generated.
\end{proof}

For the class of two-sided skew braces, the following lemma also provides a description of the ideal generated by one element.

\begin{lem}
\label{lem:clausura-ideal-twosided}
Let $B$ be a two-sided skew brace and let $u\in B$. Take $S = \langle u \rangle_+$. Then, 
\[ u^B = \langle b + \lambda_a({}^cS) - b\mid a,b,c\in B\rangle_+\] 
\end{lem}

\begin{proof}
Call $\mathcal{R} := \{b+\lambda_a({}^c S)-b\mid a,b,c\in B\}$. Since in two-sided skew braces, multiplicative conjugation is a skew brace automorphism, it suffices to check that $\alpha(S)\in \mathcal{R}$, for every $\alpha$ being an additive conjugation in~$B$, a multiplicative conjugation in~$B$, or a lambda map $\lambda_x$, with $x\in B$.

Take $a,b,c\in B$, and $S' = b+\lambda_a({}^c S)-b$. Suppose that $\alpha$ is an additive conjugation by an element $x \in B$. Then, it holds that
\[ \alpha(S') = x+b+\lambda_a({}^c S)-b-x = (x+b)+\lambda_a({}^c S)-(x+b)\in \mathcal{R}.\]
Suppose that $\alpha$ is a multiplicative conjugation by an element~$x$. Since, $B$ is two-sided, $\alpha$ is also an automorphism of the additive group. Then, it holds that
\[ \alpha(S') = {}^x(b+\lambda_a({}^c S)-b) ={}^xb+\lambda_{{}^xa}({}^{xc} S)-{}^xb\in \mathcal{R},\]
as ${}^x(\lambda_a(y)) = {}^x(-a + ay) = -{}^xa+{}^x(ay)=-{}^xa+{}^xa {}^xy$ for every $y\in {}^cS$.
Suppose that $\alpha = \lambda_x$, for some $x\in B$. Then, it holds that
\[ \alpha(S') = \lambda_x(b+\lambda_a({}^c S)-b) = \lambda_x(b)+\lambda_{xa}({}^c S)-\lambda_x(b)\in \mathcal{R}. \qedhere \]
\end{proof}

\begin{theo}
\label{teo:Hall-soluble-TWO-SIDEDskewbrace}
Every i-noetherian soluble two-sided skew brace is finitely generated.
\end{theo}

\begin{proof}
We proceed as in the proof of Theorem~\ref{teo:Hall-soluble-BRACE}, so that we can write $B = IS$, where $S = \langle s_1, \dots, s_m\rangle $ and $I$ is an abelian ideal such that $I = a_1^B + \dots + a_k^B$ for some $a_1, \dots, a_k \in I$. In this case, it suffices to check that 
\[ a_i^B \leq \langle a_i, s_1,\cdots, s_m\rangle, \quad \text{for every $1\leq i \leq k$.}\]
Let $1\leq i\leq k$ and call $U = \langle a_i \rangle_+$. Take $a,b,c\in B$, and write $a = t_1x_1$, $b = t_2 + x_2$ and $c = t_3x_3$ for some $x_1,x_2,x_3\in I$ and $t_1,t_2,t_3 \in S$. It follows that
\[ b + \lambda_a({}^c U) - b = t_2 + x_2 + \lambda_{t_1x_1}({}^{t_3x_3}U) - x_2 - t_2 = t_2 + \lambda_{t_1}({}^{t_3}U) - t_2\]
as $x_1, x_2, x_3 \in I$, $U \leq I$ and $I$ is an abelian skew brace. Therefore,
\[ a_i^B = \langle t_2 + \lambda_{t_1}({}^{t_3}U) - t_2\mid t_1, t_2, t_3 \in S\rangle \leq \langle a_i, s_1, \cdots, s_m\rangle\qedhere\]
\end{proof}

McLain's theorem \cite[Theorem~3.2]{McLain56} asserts that for locally nilpotent groups, we can go a step further in the previous Hall's result: the maximal condition on normal subgroups is equivalent to the maximal condition on subgroups, and therefore both are equivalent to the group being finitely generated. We study an analogous result for multipermutability and central nilpotency in skew braces, which are nilpotent notions in skew braces with a strong impact on the classification of solutions of the YBE.

\begin{lem}
\label{lem:fact-rightseries}
Let $B$ be a skew brace such that $B=(B\ast B)S$ for some subbrace $S$ of~$B$. Then, $B=B^{(n)}S$ for any $n\in\mathbb{N}$.
\end{lem}

\begin{proof}
Let $n > 2$ and take the series
\[ B^{(n)} \leq B^{(n-1)} \leq \dots \leq B^{(2)} = B\ast B \leq B .\]
Without loss of generality, we may assume $B^{(n)} = 0$ (otherwise, we could work with $\bar{B} = B/B^{(n)}$). Then, we know that 
\[ B^{(n-1)} \leq \operatorname{Soc}(B),\  \ldots,\  B^{(2)} = B\ast B \leq \operatorname{Soc}_{n-2}(B).\] 
Thus, $B = S(B\ast B)= S\big((B\ast B)\cap \operatorname{Soc}_{n-2}(B)\big)$.

Assume that $B = S\big(\operatorname{Soc}_{n-2-k}(B) \cap (B\ast B)\big)$ for some $0 \leq k < n-2$. We proceed to see that $B = S\big(\operatorname{Soc}_{n-2-(k+1)}(B) \cap (B\ast B)\big)$. If $b_1$, $b_2 \in B$, we can take $s_i$, $s_i' \in S$, $z_i$, $z_i'\in \operatorname{Soc}_{n-2-k}(B)$ such that $b_i = s_i+z_i= s_i'z_i'$, $i = 1$, $2$. Thus, applying~\eqref{eq:ast_dist_sum} and~\eqref{eq:ast_dist_prod}
\begin{align*}
b_1 \ast b_2 & = (s_1'z_1') \ast (s_2+z_2) = (s_1'z_1') \ast s_2 + s_2 +(s_1'z_1') \ast z_2 -s_2= \\
& = s_1'\ast(z_1'\ast s_2) + z_1'\ast s_2 + s_1' \ast s_2 + s_2 +(s_1'z_1')\ast z_2 -s_2
\end{align*}
Therefore, we can write $b_1 \ast b_2 = s_1' \ast s_2 + w$, with $w \in \operatorname{Soc}_{n-2-k-1}(B)$, as  $z_1' \ast s_2, (s_2+(s_1'z_1')\ast z_2-s_2) \in \operatorname{Soc}_{n-2-k-1}(B)$. 
Indeed, we have that 
\[ w \in \operatorname{Soc}_{n-2-k-1}(B) \cap (B\ast B),\ \text{as $b_1\ast b_2, s'_1\ast s_2\in B\ast B$.}\]
Therefore, we conclude that $B\ast B \leq S\big(\operatorname{Soc}_{n-2-k-1}(B) \cap (B\ast B)\big)$. Hence, it follows that
\begin{align*}
B & = S\big(\operatorname{Soc}_{n-2-k}(B) \cap (B\ast B)\big) \leq S(B\ast B) \\
& \leq  S\big(S\operatorname{Soc}_{n-2-k-1}(B) \cap (B\ast B)\big) \\
& = S\big(\operatorname{Soc}_{n-2-k-1}(B) \cap (B\ast B)\big).
\end{align*}
We conclude by induction that $B=S\big(\operatorname{Soc}_0(B)\cap (B\ast B)\big)=S$.
\end{proof}

\begin{theo}
\label{teo:McLain-local-multiperm}
Let $(B,+,\cdot)$ be a locally multipermutational skew brace. If $B$ is i-noetherian, then $B$ is finitely generated.
\end{theo}

\begin{proof}
Only one implication is in doubt. Assume that $B$ is i-noetherian. Since the maximal condition on ideals is preserved by quotients, $B/B\ast B$ is a trivial skew brace which is also i-noetherian. Thus, $B/B\ast B$ as a trivial skew brace is isomorphic to a finitely generated nilpotent group. Therefore, we can write $B = S(B\ast B)$ for some finitely generated subbrace $S = \langle s_1, \ldots, s_m\rangle$ with $s_1, \ldots, s_m\in B$. Applying Lemma~\ref{lem:fact-rightseries}, we have that $B = SB^{(n)}$ for every $n\in \mathbb{N}$. 

By hypothesis, $S$ has finite multipermutational level. Suppose that $n$ is the multipermutational level of~$S$. Since $B/B^{(n+2)}$ is isomorphic to $S/(S \cap B^{(n+2)})$, we have that $B/B^{(n+2)}$ has finite multipermutational level at most~$n$. In particular, $J:=B^{(n+1)} =B^{(n+2)}$, which means that $J\ast B = J$.

If $J = 0$ then we are done, as in this case, $B = SJ= S$ is finitely generated. Otherwise, since $B$ is i-noetherian, we can take a maximal ideal $L$ of $B$ contained in~$J$. Thus, $J/L$ is a chief factor of $B$. Arguing as in~\cite[Theorem~4.6]{BallesterEstebanFerraraPerezCTrombetti25-pacificjm-cent-nilp} for the socle of a skew brace, it follows that $J/L \leq \Soc(B/L)$. Hence, it holds that $J = J\ast B \leq L$ which is strictly contained in~$J$, a contradiction.  
\end{proof}

\begin{cor}
\label{cor:McLain-local-central}
    Let $B$ be a locally centrally nilpotent skew brace. The following are equivalent:
    \begin{enumerate}
        \item $B$ is i-noetherian;
        \item $B$ is finitely generated and it is centrally nilpotent;
        \item $B$ is s-noetherian.
    \end{enumerate}
\end{cor}

\begin{proof}
We just need to prove that if $B$ is a finitely generated centrally nilpotent skew brace, then $B$ is s-noetherian. We can adapt the argument of Jennings~\cite{Jennings44} for nilpotent rings and groups. 

We argue by induction on the centrally nilpotent class. Skew braces of centrally nilpotent class $1$ are abelian and so the result is trivially true for them as all subbraces are ideals. Assume now that the implication holds for centrally nilpotent skew braces with classes less than $c > 1$. Let $B$ be a centrally nilpotent skew brace of class~$c$. Let 
\[S_1\subseteq S_2\subseteq\dots\subseteq S_m\subseteq\dotsb\]
be a chain of subbraces of~$B$. Let $\bar B=B/\gamma_{c}(B)$. We have that $\bar B$ is a centrally nilpotent brace of class~$c-1$. Let $\bar S_i=S_i\gamma_c(B)/\gamma_c(B)$, we have that
\[\bar S_1\subseteq \bar S_2\subseteq \dots\subseteq \bar S_m\subseteq \dotsb\]
is a chain of subbraces of $\bar B$. By induction, there exists an integer $j$ such that $\bar S_i=\bar S_j$ for all $i\ge j$. This implies that $S_i+\gamma_c(B)=S_j+\gamma_c(B)$ for all $j\ge i$. Now consider $T_i=S_i\cap \gamma_c(B)$. By \cite[Proposition~4.2]{BallesterEstebanFerraraPerezCTrombetti25-pacificjm-cent-nilp}, $T_i\subseteq \gamma_c(B)\subseteq \zeta(B)$. Hence the subbrace $T_i$ turns out to be an ideal of~$B$. Since $\zeta(B)$ is also s-noetherian and
\[T_1\subseteq T_2\subseteq\dots\subseteq T_m\subseteq \dotsb,\]
there exists a natural number $k$ such that $T_i=T_k$ for all $i\ge k$. Let $l:=\max\{j, k\}$. For $i\ge l$, we have that
    \[S_i=S_i\cap S_i\gamma_c(B)=S_i\cap S_l\gamma_c(B)=S_l(S_i\cap \gamma_c(B))=S_l(S_l\cap \gamma_c(B))=S_l.\]
Hence, we conclude that $B$ is also s-noetherian.
\end{proof}

This corollary does not hold for locally multipermutational skew braces. In fact, if a skew brace is s-noetherian, then all its subbraces are finitely generated. The next example shows that there are finitely generated skew braces with finite multipermutational level which are not s-noetherian.

\begin{ex}
\label{ex:skbr}
Consider an infinite direct sum of infinite cyclic groups, $(G,+)=\bigoplus_{n\in \mathbb{Z}}(\mathbb{Z},+)$. Take $(B,+) = G \oplus \langle x\rangle_+$, where $\langle x\rangle_+$ is also an infinite cyclic group, and define 
$$\lambda\colon \langle x\rangle_+\longrightarrow\operatorname{Aut}(G,+),$$ 
such that $\lambda_x(1_n)=1_{n+1}$ for every $n\in\mathbb{Z}$.

In $(B,+)$, we can construct a skew brace by taking $(B,\cdot)=(G,+)\rtimes_{\lambda}\langle x\rangle$ and the lambda function coincides with the function defined earlier. Then, we can see that $B$ is finitely generated: for instance, $B = \langle 1_0, x\rangle$. It is also easy to check that $\operatorname{Soc}(B)=G$, which is clearly not finitely generated. Thus, $B$ is not s-noetherian but it is multipermutational of level $2$, as $B/G$ is an abelian skew brace.
\end{ex}

In \cite{McLain56}, McLain has also investigated the connection between the minimal condition on subgroups and the minimal conditions on normal subgroups in the setting of locally nilpotent groups. His result depends on a key result in groups, whose skew brace analogue is still an intriguing open problem. Recall that in~\cite{BallesterEstebanFerraraPerezCTrombetti25-pacificjm-cent-nilp}, we say that a subbrace $S$ of a skew brace $B$ \emph{has finite index} if both indices $|(B,+):(S,+)|$ and $|(B,\cdot):(S,\cdot)|$ are finite. Furthermore, in~\cite{VanAntwerpen2026} the authors prove that in that case both indices coincide.

\begin{quest}
\label{quest:index}
Given a subbrace $S$ of a skew brace $B$ with finite index, does $S_B$ have finite index too?
\end{quest}

We are able to show that this question is affirmative in some cases. Firstly, we consider some restrictions on the socle and central series of a skew brace.

\begin{prop}
\label{prop:fingenindex}
Let $(B,+,\cdot)$ be a skew brace such that $(B,\cdot)$ is finitely generated and there is a natural number $n$ such that $\operatorname{Soc}_{n}(B)$ has finite index in $B$, then every subbrace of finite index contains an ideal of finite index.
\end{prop}

\begin{proof}
Let $S$ be a subbrace of finite index and suppose without loss of generality that $S_B=0$. Observe that for every $a\in \Soc(B)$ and every $b\in B$ it follows that
\[ bab^{-1} = b(a+b^{-1}) = ba - b = b + \lambda_b(a) - b = \lambda_b(a)\]
as $\Soc(B)$ is an $\lambda$-invariant. Thus, $I = \mathrm{Core}_{(B,\cdot)}((S,\cdot)) \cap \operatorname{Soc}(B)$ is an ideal, as $I$ is $\lambda$-invariant (i.e. $B\ast I \subseteq I$), $(I,+)$ is normal in $(B,+)$ and $I \ast B=0$. Therefore, by assumption $I$ must be $0$, and since $\mathrm{Core}_{(B,\cdot)}((S,\cdot))$ has finite index in $(B,\cdot)$, it follows that $\operatorname{Soc}(B)$ is finite. Moreover, since $\operatorname{Soc}(B)$ is finite, every factor of the upper socle series has finite exponent (with a similar proof of \cite[Theorem~4.16]{BallesterEstebanFerraraPerezCTrombetti25-pacificjm-cent-nilp}). 

Now, $\operatorname{Soc}_n(B)$ is multiplicatively finitely generated as it has finite index in $B$, thus $\operatorname{Soc}_n(B)/\operatorname{Soc}_{n-1}(B)$ is finite and $\operatorname{Soc}_{n-1}(B)$ has finite index. Repeating the same argument, we obtain that every socle factor of $B$ is finite, and hence, $B$ is finite.
\end{proof}

\begin{cor}
Let $(B,+,\cdot)$ be a skew brace such that $(B,\cdot)$ is finitely generated and such that the hypercentre $\zeta_{\infty}(B)$ has finite index in $B$, then every subbrace of finite index contains an ideal of finite index.
\end{cor}

\begin{proof}
The same argument applies as a finitely generated hypercentral skew brace is centrally nilpotent.
\end{proof}

\begin{rem}
Recall that the structure skew brace $G(X,r)$ of a finite solution of the YBE $(X,r)$ (i.e. $X$ is a finite set) satisfies that $G(X,r)/\Soc(G(X,r))$ is finite, as $(G(X,r),\cdot)/\Soc(G(X,r))$ is a quotient of the permutation skew brace $\mathcal{G}(X,r)$, which is finite. Hence, Question~\ref{quest:index} holds for~$G(X,r)$.
\end{rem}

In \cite{VanAntwerpen2026} the authors show that Question~\ref{quest:index} holds for two-sided skew braces. In our approach we use a similar argument of Lemma~\ref{lem:clausura-ideal-twosided}. This result allows us to extend McLain's result about minimal conditions within the class of two-sided skew braces. 


\begin{lem}
\label{lem:core_two-sided}
Let $B$ be a two-sided skew brace and let $S$ be a subbrace of~$B$. Then,
\[  S_B = \bigcap\{b+\lambda_a({}^c S)-b\mid a,b,c\in B\}\]
\end{lem}

\begin{proof}
The proof is analogous to the proof of Lemma~\ref{lem:clausura-ideal-twosided}.
\end{proof}

\begin{rem}
In two-sided skew braces, equations~\eqref{eq:ast_dist_sum} and~\eqref{eq:ast_dist_prod} also hold on the right: for every $a,b,c\in B$,
\begin{align}
(a+b)\ast c & = -b + a\ast c + b + b\ast c, \label{eq:astright_dist_sum} \\
a\ast (bc) & = a\ast c + a\ast b + (a\ast b)\ast c. \label{eq:astright_dist_prod}
\end{align}
Since multiplicative conjugations are also automorphisms of the additive group, for a given $x\in B$ and a subbrace $H\leq B$, it turns out that 
$$C_H(x)=\{b\in H\mid [b,x]_+=[b,x]_\cdot = b\ast x = x\ast b = 0\}$$
is a subbrace of~$B$.
\end{rem}

We are in a position to prove McLain's result for two-sided skew braces.

\begin{theo}
\label{theo:minid}
Let $B$ be a locally multipermutational two-sided skew brace. If $B$ is i-artinian, then it is s-artinian. Moreover, $B$ is a hypercentral and soluble skew brace.
\end{theo}

\begin{proof}
Recall that multipermutational skew braces are right nilpotent of nilpotent type. In two-sided skew braces of nilpotent type, right nilpotency is equivalent to central nilpotency (see \cite[Theorem~4.26]{Trappeniers2023two}), then we can assume that $B$ is locally centrally nilpotent.

Besides, in locally centrally nilpotent skew braces, minimal ideals are always contained in the center (\cite[Theorem~4.5]{BallesterEstebanFerraraPerezCTrombetti25-pacificjm-cent-nilp}). Therefore, $B$ is hypercentral as it is i-artinian. Moreover, it is possible to find a minimum ideal of finite index~$F$, and according to \cite[Corollary~3.5]{VanAntwerpen2026}, we can say that $F$ is contained in every subbrace of finite index.

Let $Z=\zeta(F)_B$, which is an abelian ideal of~$B$. Note that $Z$ cannot be~$0$, because the set $\{I\lhd B\mid I\subseteq F\}$ has a minimal element $M$, which is also a minimal ideal of~$B$, so $M\leq \zeta(B)\cap F\leq \zeta(F)$. Moreover, $Z$ can be 
seen isomorphic to an abelian group. Take $T$ its torsion group. Recall that for two-sided skew braces, characteristic subgroups of the additive group of an ideal are also ideals. Thus, $T$ is an ideal of~$B$ such that $Z/T$ is isomorphic to a torsion-free abelian group. 
Since $B$ is i-artinian, take $I/T$ a minimal ideal of $B/T$ contained in~$Z/T$. Once again, $I/T\leq \zeta(B/T)$ by \cite[Theorem~4.5]{BallesterEstebanFerraraPerezCTrombetti25-pacificjm-cent-nilp}. Thus, $I/T$ is torsion free and all additive subgroups of $I/T$ are ideals of $B/T$. This contradicts the minimality of~$I/T$.

Therefore, $Z$ is isomorphic to a periodic abelian group, and $\pi(Z)$ must be finite, as $B$ is i-artinian, and characteristic subgroups of $Z$ are ideals of~$B$. We claim that for each additive order $n$ there are only a finite number of elements of $Z$ with order~$n$. Indeed, it suffices to prove it for each prime $p\in \pi(Z)$. Let $p\in\pi(Z)$ and consider $P=Z[p]$, the subgroup generated by all elements of order~$p$. Thus, $P$ is an ideal of~$B$. If $(P,+)$ is an infinite elementary abelian $p$-group, write $(P,+)=\langle a\rangle \oplus K$ for some $a\in P$. Since $K\leq \zeta(F)$ and $F$ has finite index in $B$, the set $\{b+\lambda_a({}^c K)-b\mid a,b,c\in B\}$ is finite and each of its elements has also finite index in $P$. Therefore, $K_B=P_1 \leq K$ has still finite index in $P$, and therefore, it is an infinite elementary abelian $p$-group. The same process could be iteratively repeated, so that it yields an infinite descending series of ideals of~$B$, which is a contradiction. Thus, $Z$ is an abelian skew brace isomorphic to a periodic abelian group with finite rank, and $\pi(Z)$ is also finite. Hence, $Z$ is s-artinian.

Now, let $Z'$ be the ideal such that $Z'/Z=(\zeta(F/Z))_{B/Z}$. Clearly $Z'/Z$ is again a periodic abelian skew brace of finite rank with $\pi(Z'/Z)$ finite and $Z<Z'$. Recursively, this yields an ascending series of ideals in $B$, $\{Z^{(\alpha)}\}_{\alpha\in\mathcal{O}}$, that is central in~$F$. 

Let $x\in Z'$ and $b\in F$. Since $Z'/Z = (\zeta(F/Z))_{B/Z}$, it follows that 
\[ -b+x+b=x+u, \quad \lambda_b(x)=x+u', \quad  b^{-1}xb=x+u'',\]
for some $u,u', u''\in Z$. Moreover, each of these skew brace conjugates of $x$ in $F$ preserve the additive order of $x$, so that the additive order of $u$, $u'$ and $u''$ must be a divisor of the order of~$x$. But in $Z$ there are only a finite amount of elements of a fixed order (and also its divisors), and therefore, $x$ has only a finite number of these skew brace conjugates. Thus, $C_F(x)$ is a subbrace such that both its additive and multiplicative subgroup have finite index in $F$ and so also in~$B$, that is $C_F(x)$ has finite index as a subbrace in~$B$. Thus, $F=C_F(x)$, and so $x\in \zeta(F)$. Hence, $Z'\leq \zeta(F)$, and since $Z'$ is an ideal contained in $\zeta(F)$, $Z=Z'$, which is only possible if $Z=Z'=\zeta(F)=F$.

Finally, we conclude that $F$ is an abelian skew brace which is s-artinian, and $B/F$ is a finite centrally nilpotent skew brace. Hence, $B$ is s-artinian and a soluble skew brace.
\end{proof}

\section{Finiteness conditions on solutions of the YBE}
\label{sec:finite-cond-sol}

The purpose of this section is to expand the theory of finiteness conditions of skew braces to the setting of solutions of the YBE. We shall delve into interrelations between finiteness conditions of solutions and its associated skew braces.

We start by introducing finite generation for solutions.

\begin{defi}
Let $(X,r)$ be a solution and $Y\subseteq X$. Consider 
$$ \bar{Y}=\bigcap\{S\subseteq X \mid Y\subseteq S \ \text{and}\  (S,r)\leq (X,r)\}.$$ 
We say that $(\bar Y, r)$ is the {\it solution generated by}~$Y$. If $Y$ is finite and $(X,r) = (\bar Y, r)$, we say that $(X,r)$ is \emph{finitely generated}.
\end{defi}

The following result shows that finite generation is inherited by the structure and permutation skew brace of a solution.
\begin{prop}
\label{prop:fin-gen_sol->fin-gen_skew}
If $(X,r)$ is a solution such that $X=\bar{Y}$, then $G(X,r)$ is generated by $Y$ as a skew brace. Consequently, if $(X,r)$ is a finitely generated solution, then $G(X,r)$ is a finitely generated skew brace, and so $\mathcal{G}(X,r)$ is.
\end{prop}

\begin{proof}
Recall that $G(X,r)=\langle X \mid xy=\lambda_x(y)\rho_y(x), \ \forall\, x,y \in X\rangle$. The proof assumes that $(X,r)$ is injective, i.e. $\iota\colon x\in X \mapsto x\in G(X,r)$ is injective,  as in general, $G(X,r)\simeq G(\operatorname{Inj}(X,r))$. Call $\lambda$ and $\lambda^G$, $\rho$ and $\rho^G$, respectively, the $\lambda$ and $\rho$ maps in $(X,r)$ and $G(X,r)$.
 
Let $Y$ be a subset of $X$ such that $\overline{Y}=X$, and consider $H=\langle Y\rangle $ the subbrace generated by $\iota(Y)$ in~$G(X,r)$. It suffices to show that $Z:=\{x\in X\mid \iota(x) \in H\}$ provides a subsolution $(Z,r) \leq (X,r)$, as $Z$ clearly contains~$Y$, and therefore, it follows that $Z = X$ and so $H = G(X,r)$. 

Let $x, y \in Z$. Since $\iota(x),\iota(y) \in H$, by~\eqref{eq:commuting_diagram_X-G(X,r)}, it holds that
\[ \iota(\lambda_x(y)) = \lambda^G_{\iota(x)}(\iota(y)) \in H \]
Hence, $\lambda_x(Z) \leq Z$. On the other hand, let $x, z\in Z$ and $y \in X$ such that $\lambda_x(y) = z$. Thus, $y = \lambda_x^{-1}(z)$. Since $r$ is bijective, once again by~\eqref{eq:commuting_diagram_X-G(X,r)}, it follows that
\[ \iota(y) = \iota(\lambda_x^{-1}(z)) = (\lambda^G)^{-1}_{\iota(x)}(\iota(z)) = \lambda^G_{\iota(x)^{-1}}(\iota(z)) \in H\] 
Therefore, $\lambda_x(Z) = Z$ for every $x\in Z$. Analogously, $\rho_x(Z) = Z$ for every $x\in Z$. Hence, $r(Z\times Z) = Z\times Z$ and $(Z,r)\leq (X,r)$ as we wanted to prove.
\end{proof}

We do not know if the converse of this proposition holds for structure skew braces of solutions, but certainly, it does not hold for permutation skew braces of solutions.

\begin{ex}
Let $X=\mathbb{Z}$ and take the map $r \colon \mathbb Z \times \mathbb Z \rightarrow \mathbb Z \times \mathbb Z$, given by $r(n,m)=(-m,-n)$. Clearly, $(X,r)$ is a Lyubashenko solution. Observe that in $(X,r)$, for every finite subset $X\subseteq \mathbb{Z}$, $(\bar{X}, r)$ is given by $\bar X = X\cup \{-x\mid x\in X\}$. Hence, finitely generated subsolutions are finite, and therefore, $(X,r)$ is not finitely generated. Nevertheless, $\lambda_x = -\mathrm{id}_{\mathbb{Z}} = \rho_x$ for every $x\in \mathbb{Z}$. Hence, $\mathcal{G}(X,r) = \langle (-\mathrm{id}_{\mathbb Z}, - \mathrm{id}_{\mathbb{Z}}) \rangle$ is finitely generated.
\end{ex}

\begin{rem}
Recall that if $B$ is a skew brace, then $r_B(a,b) = (\lambda_a(b), \rho_b(a))$ is provided by the $\lambda$ and $\rho$-actions in~$B$. Thus, the subsolution generated by a subset $S\subseteq B$ is always contained in the subbrace generated by $S$. Therefore, if $(B,r_B)$ is finitely generated, so is $B$ as a skew brace.

The opposite does not hold, not even in the multipermutational case. For instance, if $B$ is an infinite abelian skew brace isomorphic to a finitely generated abelian group, then $(B,r_B)$ is a twist solution.  Thus, every subset of $B$ is a subsolution, and therefore, it is not finitely generated.
\end{rem}

\bigskip

We turn now to introduce locality conditions for solutions of the YBE.

\begin{defi}
Let $(X,r)$ be a solution and let $\mathfrak{Y}$ be a class of solutions of the YBE. A solution is called \textit{locally} $\mathfrak{Y}$, if every finite subset is contained in a subsolution in $\mathfrak{Y}$.
\end{defi}

\begin{rem}
If $\mathfrak{Y}$ is a class of solutions which is closed with respect to subsolutions (e.g. finite solutions or (hyper)multipermutational solutions), then being locally $\mathfrak{Y}$ means that all its finitely generated subsolutions are in $\mathfrak{Y}$. For instance, $(X,r)$ is \emph{locally finite} if all its finitely generated subsolutions are finite.
\end{rem}

\begin{prop}
\label{prop:Blocal-rBlocal}
Let $\mathfrak{X}$ and $\mathfrak{Y}$ be a class of skew braces, and a class of solutions, respectively. Assume that for every skew brace $B$, $B\in \mathfrak{X}$ if and only if $(B,r_B)\in\mathfrak{Y}$. It follows that if $B$ is locally $\mathfrak{X}$, then $(B,r_B)$ is locally $\mathfrak{Y}$.
\end{prop}

\begin{proof}
Assume that $F$ is a finite subset of $B$, such that it is contained in a subbrace $H\in \mathfrak{X}$. Then, $(H,r_H)\in \mathfrak{Y}$, and the proposition holds.
\end{proof}

Note that the converse is false, as every abelian skew brace $B$, which is isomorphic to a non periodic abelian group, satisfies that $(B,r_B)$ is a twist solution, and therefore, it is locally finite. Nevertheless, the following equivalence holds in the case of solutions and associated permutation skew braces.

\begin{theo}
\label{teo:Xlocall<->G(X,r)locall}
Let $\mathfrak{X}$ be a class of skew braces closed with respect to subbraces and images, and let $\mathfrak{Y}$ be a class of solutions, such that for any solution $(X,r)\in\mathfrak{Y}$ if and only if $\mathcal{G}(X,r)\in\mathfrak{X}$. Then, $(X,r)$ is locally $\mathfrak{Y}$ if and only if $\mathcal{G}(X,r)$ is locally $\mathfrak{X}$.
\end{theo}

\begin{proof}
Suppose that $(X,r)$ is a locally $\mathfrak{Y}$ solution and let $F$ be a finite subset of $\mathcal{G}(X,r)$, then each element of $F$ is a finite product of elements of the type $(\lambda_x,\rho_x^{-1})$, which are induced by a finite number of elements of $X$. The latter are contained in a subsolution $Y\in \mathfrak{Y}$. Thus, $\mathcal{G}(Y,r_Y)\in \mathfrak{X}$. Moreover, $\langle (\lambda_y, \rho_y^{-1})\mid y\in Y\rangle \in \mathfrak{X}$, as it is a quotient of $\mathcal{G}(Y,r_Y)$ by Proposition~\ref{prop:subsol-subbrace}, and $F \subseteq\langle (\lambda_y, \rho_y^{-1}) \mid y\in Y\rangle$.

Suppose that $\mathcal{G}(X,r)$ is a locally $\mathfrak{X}$ skew brace, and let $Y$ be a finitely generated subsolution. By Proposition~\ref{prop:fin-gen_sol->fin-gen_skew}, $\mathcal{G}(Y,r_Y)$ is a finitely generated skew brace, that is also isomorphic to a section of $\mathcal{G}(X,r)$. Hence $\mathcal{G}(Y,r_Y)\in\mathfrak{X}$, and therefore, $Y$ is in $\mathfrak{Y}$.
\end{proof}

\begin{cor}
Let $(X,r)$ be a solution. Then, the following are equivalent:
\begin{enumerate}
\item $(X,r)$ is a locally (hyper)multipermutational;
\item $G(X,r)$ is a locally (hyper)multipermutational skew brace;
\item  $\mathcal{G}(X,r)$ is a locally (hyper)multipermutational skew brace.
\end{enumerate}
\end{cor}

Theorem~\ref{teo:Xlocall<->G(X,r)locall} cannot be used for the class of finite solutions, because there exist infinite solutions with finite permutation skew brace. However, the result for the locally finite case can still be recovered in this setting.

\begin{theo}
Let $(X,r)$ be a solution. Then, $\mathcal{G}(X,r)$ is a locally finite skew brace if and only if $(X,r)$ is a locally finite solution.
\end{theo}

\begin{proof}
If $(X,r)$ is a locally finite solution, then we can follow the same reasoning as Theorem~\ref{teo:Xlocall<->G(X,r)locall} to prove that $\mathcal{G}(X,r)$ must be locally finite.

Assume that $\mathcal{G}(X,r)$ is locally finite. If $F$ is a finitely generated subsolution of $X$, then by Proposition~\ref{prop:fin-gen_sol->fin-gen_skew}, $\mathcal{G}(F,r_{|_F})$ is finitely generated, and therefore, it is finite.

Write $F=\overline{\{y_1,\dots,y_n\}}$, for some $y_1, \dots, y_n \in X$. Then, for every $i\in\{1,\dots ,n\}$, we can define:
\begin{align*} 
A_i = & \{(\lambda_{y_{i_1}}^{\varepsilon_{i_1}}\circ \rho_{y_{j_1}}^{\varepsilon_{j_1}}\circ \dots\circ \lambda_{y_{i_m}}^{\varepsilon_{i_m}}\circ\rho_{y_{j_m}}^{\varepsilon_{j_m}})(y_i) \mid \\
 & \mid m\in\mathbb{N},\  i_t,j_t\in\{1,\dots,n\}, \  \varepsilon_{i_t},\varepsilon_{j_t}\in \{-1,0,1\}\, \forall t\in\{1,\dots,m\}\}.
\end{align*}
Since $(F, r_{|_F})$ is a solution, the second component in the equality of equation \eqref{eq:braided_eq} yields that for every $w,z\in F$,
\[ \rho_{\lambda_{\rho_{y_i}(w)}(z)}(\lambda_w(y_i))=\lambda_{\rho_{\lambda_{y_i}(z)}(w)}(\rho_z(y_i))\]
Call $z':=\lambda_{y_i}(z)$, and consider $w'\in F$ the only element such that $\rho_{z'}(w')=w$. It follows that
\[ \lambda_w(\rho_z(y_i))=\rho_{\lambda_{\rho_{y_i}(w')}(z)}(\lambda_{w'}(y_i)).\]
Therefore, we can write 
\[ A_i=\{(\overline{\rho}\circ\overline{\lambda})(y_i) \mid \overline{\rho}\in\langle \rho_{y_j} \mid j=1,\dots,n\rangle, \ \overline{\lambda}\in\langle \lambda_{y_j} \mid j=1,\dots,n\rangle \}.\]

Since $\mathcal{G}(F,r_{|_F})$ is finite, then both $\langle \rho_{y_j} \mid j=1,\dots,n\rangle$ and $\langle \lambda_{y_j} \mid j=1,\dots,n\rangle$ are finite. Hence, $A_i$ is finite, and therefore, $F=\bigcup_{i=1}^nA_i$ is also finite.
\end{proof}

We bring this section to a close by introducing and studying chain conditions for solutions.

\begin{defi}
Let $(X,r)$ be a solution. Then, we say that $(X,r)$ satisfies
\begin{enumerate}
\item $\emph{max-sub}$ (respectively $\emph{min-sub}$) if the set of its subsolutions satisfies the maximal (respectively minimal) condition;
\item $\emph{max-con}$ (respectively $\emph{min-con}$) if the set of its congruences satisfies the maximal (respectively minimal) condition.
\end{enumerate}
\end{defi}

\begin{prop}
\label{prop:maxsol}
$(X,r)$ satisfies max-sub if and only if all its subsolutions are finitely generated.
\end{prop}

\begin{proof}
Assume that $(X,r)$ is a solution satisfying $\operatorname{max-sub}$ and let $Y$ be a subsolution. By a way of contradiction, assume that $Y$ is not finitely generated. Then, it is possible to construct an ascending series of subsolutions $(\overline{\{x_1,\dots,x_n\}})_{n\in\mathbb{N}}$, which is absurd.

On the other hand, suppose that all subsolutions are finitely generated and consider  $\{Y_n\}_{n\in \mathbb{N}}$ an ascending series of subsolutions of $(X,r)$. Clearly, $\bigcup_{n\in\mathbb{N}}Y_n$ is a subsolution that is finitely generated. Hence, the previous series must stabilise at a certain number~$k$.
\end{proof}

Bearing in mind results of the previous section, the following question naturally arises:
\begin{quest}
Is max-con equivalent to max-sub for locally multipermutational solutions?
\end{quest}

We give a positive answer to this question for the class of Lyubashenko solutions. Indeed, we show that this class of solutions presents a good behaviour within chain conditions.

\emph{For the rest of the section, $(X,r)$ denotes a Lyubashenko solution, and for every $x,y\in X$, we write $r(x,y)=(f(y),g(x))$ where $f$ and $g$ are two commuting permutations of~$X$.}

\begin{lem}
\label{lem:lyu}
Let $(X,r)$ be a Lyubashenko solution and let $Y$ be a subsolution of $X$. Then, $Y$ is a union of orbits of the action of the group $\langle f,g\rangle\leq \operatorname{Sym}_X$.
\end{lem}

\begin{proof}
Let $x\in X$, and set $A=\{(f^n\circ g^m)(x) \mid n,m\in\mathbb{Z}\}$ the orbit of $x$ with respect to $\langle f,g\rangle$. Observe that $A$ must be contained in every subsolution of $(X,r)$ that contains~$x$.

Since $f$ and $g$ commute, $A$ is a subsolution as $A=\overline{\{x\}}$. Indeed, it follows that for every $y\in A$, $A=\overline{\{y\}}$. Therefore, for every subsolution $(Y,r)$, we can find a minimal subset of generators $S$ such that 
\[ Y=\overline{S}=\bigcup_{x\in S}\overline{\{x\}}. \qedhere \]
\end{proof}

\begin{cor}
\label{cor:lyusol}
    Let $(X,r)$ be a Lyubashenko solution. Then, the following are equivalent:
    \begin{enumerate}
        \item $(X,r)$ is finitely generated;
        \item $(X,r)$ satisfy $\operatorname{max-sub}$;
        \item $(X,r)$ satisfy $\operatorname{min-sub}$;
        \item $X$ has a finite number of orbits of the action of the group $\langle f,g\rangle$.
    \end{enumerate}
\end{cor}

\begin{proof}
This is a direct consequence of Lemma~\ref{lem:lyu}, as each of the four properties is equivalent to having only a finite number of subsolutions.
\end{proof}

Indeed, our next theorem claims that for max-con we can say more.
\begin{theo}
\label{teo:max-con-equiv-lyusol}
Let $(X,r)$ be a Lyubashenko solution. All properties of Corollary~\ref{cor:lyusol} are equivalent to $\operatorname{max-con}$.
\end{theo}

\begin{proof}
Suppose that $(X,r)$ satisfies the maximal condition on congruences. Let $x_0\in X$. Then 
\[ X_0=\overline{\{x_0\}}=\{(f^n\circ g^m)(x_0)\mid n,m\in\mathbb{Z}\}.\]
Let $\sim_0$ be the equivalence relation that identifies all elements in $X_0$ with~$x_0$. By definition of $X_0$, $\sim_0$ is clearly a congruence on $(X,r)$, as $f(X_0) = g(X_0) = X_0$.

If $X_0 \neq X$, take $x_1\in X\setminus X_0$ and set $X_1:=\overline{\{x_1\}}$. Once, again we consider the congruence $\sim_1$ that identifies all elements of $X_0$ and $X_1$ with elements $x_0$ and $x_1$, respectively. If $(X,r)$ is not finitely generated, iterating this process we obtain a strictly ascending chain of congruences $\{\sim_n\}_{n\in \mathbb{N}_0}$ that does not stabilise, which is impossible.

    
On the other hand, assume that $(X,r)$ is finitely generated with 
\[ X=\overline{\{x_1,\dots,x_k\}}=\overline{\{x_1\}}\cup\dots\cup\overline{\{x_k\}},\]
for certain $x_1,\dots,x_k\in X$. By a way of contradiction, we can suppose that there exists $\{\sim_n\}_{n\in \mathbb{N}}$ a strictly ascending sequence of congruences of $(X,r)$, where $\sim_1$ contains strictly $\mathrm{diag}(X)$, the diagonal of $X \times X$.

Since $X$ is finitely generated and the sequence is strictly ascending, we can assume without loss of generality that $k = 1$, i.e. $X = \overline{\{x\}}$ for some $x\in X$. Call $f_h$, $g_h$ the permutations defined by the solution $X/{\sim_h}$ for every $h\in\mathbb{N}$. Let $y \in X$ with $y\neq x$ and $x \sim_1 y$. Since $X$ is one-generated, there exist $n_1,m_1\in \mathbb{Z}$, not both equal to zero, such that $y = (f^{n_1}\circ g^{m_1})(x)$. Thus, $x \sim_1 (f^{n_1}\circ g^{m_1})(x)$, and by definition of congruence, 
for every $s,t\in \mathbb{Z}$, 
\[ (f^s \circ g^t)(x) \sim_1 (f^{s+n_1}\circ g^{t+m_1})(x).\]
which means that $f_1^{n_1}\circ g_1^{m_1}=\operatorname{id}_{X/{\sim_1}}$. In general, that means that whenever we find that $x\sim_h (f^n\circ g^m)(x)$, then $f_h^n\circ g_h^m=\operatorname{id}_{X/{\sim_h}}$.

Define the subgroup $R_h=\langle a^n b^m \mid x\sim_h (f^n\circ g^m)(x)\rangle$ of the free abelian group of rank $2$, $F_{\mathrm{ab}}(\{a,b\}) = \langle a, b \rangle$. Then, for every $h\in\mathbb{N}$, it holds that
\[ \langle a,b\rangle/R_h \cong \langle f_h,g_h\rangle\leq \mathrm{Sym}_{X/\sim_h} \times \mathrm{Sym}_{X/\sim_h}.\]

But, if $\{\sim_n\}_{n\in \mathbb{N}}$ is strictly ascending, so is $\{R_h\}_{h\in \mathbb{N}}$, which is a contradiction because $\langle a,b\rangle$ satisfies the maximal condition.  
\end{proof}

The situation for min-con is different as we show in the following result.
\begin{theo}
\label{teo:min-con-lyusol}
Let $(X,r)$ be a Lyubashenko solution satisfying $\operatorname{min-con}$. Then, $(X,r)$ is finite.
\end{theo}

\begin{proof}
Let $x_0\in X$ and define $\sim_0$ as the equivalence relation that identifies all elements in $X\setminus \overline{\{x_0\}}$. Clearly, it is a congruence as $f(X\setminus \overline{\{x_0\}}) = g(X\setminus \overline{\{x_0\}}) = X\setminus \overline{\{x_0\}}$. If $\sim_0\neq \operatorname{diag}(X)$, then we can take an element $x_1\in X\backslash\overline{\{x_0\}}$, and analogously, define $\sim_1$ as the congruence that identifies all elements in $X\backslash\overline{\{x_0,x_1\}}$. Iterating this process, we obtain a descending sequence of congruences $\{\sim_n\}_{n\in\mathbb{N}_0}$ that stabilises at some $n_0 \in \mathbb{N}_0$. This means that $X = \overline{\{x_0,x_1, \dots, x_{n_0}\}}$, i.e. $(X,r)$ is finitely generated.

By Lemma~\ref{lem:lyu}, without loss of generality we can suppose that 
\[ X=\overline{\{x\}}=\{(f^n\circ g^m)(x)\mid n,m\in\mathbb{Z}\}.\]
Moreover, if both $f$ and $g$ have finite order, $X$ is clearly finite. Thus, without loss of generality assume also that $f$ has infinite order.

For every $n\in \mathbb{N}$, $f^{2^n}(x)\neq x$. Otherwise, there will be a positive integer $m$ such that 
\[ f^{2^m}((f^h\circ g^k)(x))=(f^h\circ g^k)(f^{2^m}(x))=(f^h\circ g^k)(x),\]
for every $h$, $k \in \mathbb{Z}$. Hence, $f^{2^m}=\operatorname{id}_X$, a contradiction.

For every $m\in \mathbb{N}$, we define the equivalence relation $\sim_m$ as follows:
\[ y\sim_m z \quad \Longleftrightarrow \quad \exists\, k\in\mathbb{Z} \ \text{such that}\ f^{k2^m}(y)=z.\]
Since $f$ and $g$ commutes it is easy to verify that $\sim_m$ is a congruence for every $m\in \mathbb{N}$. Moreover, by definition, $\{\sim_m\}_{m\in\mathbb{N}}$ forms a strictly descending sequence of congruences on $(X,r)$ that does not stabilise. Hence, we arrive to a final contradiction.
\end{proof}

Furthermore, our last example shows that finitely generated Lyubashenko solutions does not necessarily satisfy min-con.

\begin{ex}
\label{ex:lyusol-notmin}
Let $X=\mathbb{Z}\times\mathbb{Z}$, and consider the permutations $f((x_1,x_2))=(x_1+1,x_2)$ and $g((x_1,x_2))=(x_1,x_2+1)$. Clearly, $f$ and $g$ commute, so that $(X,r)$ is a Lyubashenko solution with $r(\mathbf{x}, \mathbf{y}) = (f(\mathbf{y}),g(\mathbf{x}))$, for every $\mathbf{x,y}\in X$. Moreover, $(X,r)$ is clearly one-generated as $(X,r) = (\overline{\{(0,0)\}}, r)$. Nevertheless, $(X,r)$ does not satisfy $\operatorname{min-con}$ as it is not finite.
\end{ex}
\section*{Acknowledgements}
A part of this research has been developed during a stay of the second, the third, and the fourth authors in the \emph{Centre International de Rencontres Math\'ematiques} (Aix-Marseille Universit\'e, Centre National de la Recherche Scientifique and Soci\'et\'e Math\'ematique de France) in the framework of the programme \emph{Recherche en r\'esidence}. We thank this centre for their support and their kindness during this stay. The second, the third, and the fourth authors have been supported by the grant PID2024-159495NB-I00, funded by MICIU/AEI/10.13039/501100011033 and by ERDF/EU. The second and the fourth authors are supported by the grant CIAICO/2023/007 from the
Conselleria d’Educaci\'o, Universitats i Ocupaci\'o, Generalitat Valenciana. 
Another part of this project has been developed during a sojourn of the first author at the \emph{Department de Matemàtiques} of the \emph{Universitat de València}, supported by a grant part of the Erasmus+ for Traineeship 2025/2026 Programme of the Università degli Studi della Campania "Luigi Vanvitelli", with code 2025-1-IT02-KA131-HED-000311161. The first author would like to thank the \emph{Department de Matemàtiques} of the \emph{Universitat de València} for the hospitality and the friendliness demonstrated during the period.
All authors are members of the nonprofit association ``Advances in Group Theory
and Applications''. We also thank Arne Van Antwerpen for his comments on the original draft.

\bibliographystyle{plain}

\end{document}